\documentclass[11pt]{article}
\usepackage{latexsym,amsmath,color,amsthm,amssymb,epsfig,graphicx,mathrsfs}
\usepackage{graphicx}
\usepackage{amssymb}
\usepackage[left=1in,top=1in,right=1in,bottom=1in]{geometry}
\usepackage[linktocpage=true]{hyperref}
\usepackage{setspace}
\usepackage{amssymb, amsmath, amsthm, graphicx,mathrsfs}

\def\qed{\hfill\ifhmode\unskip\nobreak\fi\quad\ifmmode\Box\else\hfill$\Box$\fi}
\def\ite#1{\hfill\break${}$\hbox to 50pt {\quad(#1)\hfill}}
\newtheorem{thm}{Theorem}
\newtheorem{cor}[thm]{Corollary}

\newtheorem{rem}[thm]{Remark}
\newtheorem{lem}[thm]{Lemma}

\newtheorem{claim}[thm]{Claim}

\begin{document}

\title{\vspace{-0.5in}Extensions of a theorem of Erd\H{o}s on nonhamiltonian graphs\footnote{This paper started
at SQUARES meeting of the American Institute of Mathematics (April 2016).}}

\author{
{{Zolt\'an F\" uredi}}\thanks{
\footnotesize {Alfr\' ed R\' enyi Institute of Mathematics,  Hungary
E-mail:  \texttt{furedi.zoltan@renyi.mta.hu}.
{Research was supported in part by grant (no. K116769)
from the National Research, Development and Innovation Office NKFIH,
by the Simons Foundation Collaboration Grant \#317487,
and by the European Research Council Advanced Investigators Grant 267195.}
}}
\and
{{Alexandr Kostochka}}\thanks{
\footnotesize {University of Illinois at Urbana--Champaign, Urbana, IL 61801
 and Sobolev Institute of Mathematics, Novosibirsk 630090, Russia. E-mail: \texttt {kostochk@math.uiuc.edu}.
 Research of this author
is supported in part by NSF grant DMS-1600592
and grants 15-01-05867 and 16-01-00499  of the Russian Foundation for Basic Research.
}}
\and{{Ruth Luo}}\thanks{University of Illinois at Urbana--Champaign, Urbana, IL 61801, USA. E-mail: {\tt ruthluo2@illinois.edu}.}}

\date{March 29, 2017}

\maketitle

\vspace{-0.3in}

\begin{abstract}
%
%

Let $n, d$ be integers with $1 \leq d \leq \left \lfloor \frac{n-1}{2} \right \rfloor$, and set $h(n,d):={n-d \choose 2} + d^2$. 
Erd\H{o}s proved that when $n \geq 6d$, 
each  nonhamiltonian graph $G$ on $n$ vertices with minimum degree $\delta(G) \geq d$ has at most $h(n,d)$ edges. He also provides a sharpness example $H_{n,d}$ for all such pairs $n,d$. Previously, we showed a stability version of this result: for $n$ large enough, every nonhamiltonian graph $G$ on $n$ vertices with $\delta(G) \geq d$ and more than $h(n,d+1)$ edges is a subgraph of $H_{n,d}$.

In this paper, we show that not only does the graph $H_{n,d}$ maximize the number of edges among  nonhamiltonian graphs
with $n$ vertices and minimum degree at least $d$, but in fact it maximizes the number of copies of any fixed graph $F$ when $n$ is sufficiently large in comparison with $d$ and $|F|$. We also show a stronger  stability theorem, that is, we classify all nonhamiltonian $n$-graphs with $\delta(G) \geq d$ and more than $h(n,d+2)$ edges. We show this by proving a more general theorem: we describe all such graphs with more than ${n-(d+2) \choose k} + (d+2){d+2 \choose k-1}$ copies of $K_k$ for any $k$.
\medskip\noindent
{\bf{Mathematics Subject Classification:}} 05C35, 05C38.\\
{\bf{Keywords:}} Subgraph density, hamiltonian cycles, extremal graph theory.
\end{abstract}

\section{Introduction}

Let $V(G)$ denote the vertex set of a graph $G$, $E(G)$
denote the edge set of $G$, and $e(G)=|E(G)|$. Also, if $v\in V(G)$, then $N(v)$ is the neighborhood of $v$ and $d(v)=|N(v)|$.
If $v\in V(G)$ and $D\subset V(G)$ then for shortness we will write $D+v$ to denote $D\cup \{v\}$.
For $k,t \in \mathbb N$, $(k)_t$ denotes the falling factorial $  k(k-1) \ldots (k-t+1)=\frac{k!}{(k-t)!}$.

The first Tur\'{a}n-type result for nonhamiltonian graphs was due to Ore~\cite{Ore}:

\begin{thm}[Ore~\cite{Ore}]\label{Ore}
If $G$ is a nonhamiltonian graph on $n$ vertices, then $e(G) \leq {n-1 \choose 2} + 1$.
\end{thm}

This bound is achieved only for the $n$-vertex graph obtained from the complete graph $K_{n-1}$ by adding a vertex of degree $1$.
Erd\H{o}s ~\cite{Erdos}  refined the bound in terms of the minimum degree of the graph:

\begin{thm}[Erd\H{o}s~\cite{Erdos}]\label{Erdos} Let $n, d$ be integers with $1 \leq d \leq \left \lfloor \frac{n-1}{2} \right \rfloor$, and set $h(n,d):={n-d \choose 2} + d^2$.
If $G$ is a nonhamiltonian graph on $n$ vertices with minimum degree $\delta(G) \geq d$, then
     \[e(G) \leq \max\left\{ h(n,d),h(n, \left \lfloor \frac{n-1}{2} \right \rfloor)\right\}=:e(n,d).\]
This bound is sharp for all $1\leq d\leq \left \lfloor \frac{n-1}{2} \right \rfloor$.
\end{thm}

To show the sharpness of the bound, for $n,d \in \mathbb N$ with $d \leq \left \lfloor \frac{n-1}{2} \right \rfloor$,
consider the graph $H_{n,d}$  obtained from a copy of $K_{n-d}$, say with vertex set $A$,
 by adding $d$ vertices of degree $d$ each
of which is adjacent to the same $d$ vertices in $A$.
An example of $H_{11,3}$ is on the left of Fig~\ref{fig0}.


\begin{figure}[!ht]\label{fig0}
 \centering
\includegraphics[scale=.56]{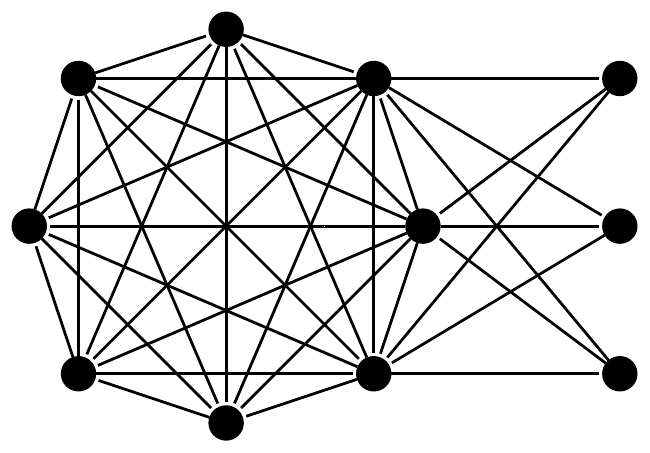}\;\;\;\hspace{2cm} \includegraphics[scale=.56]{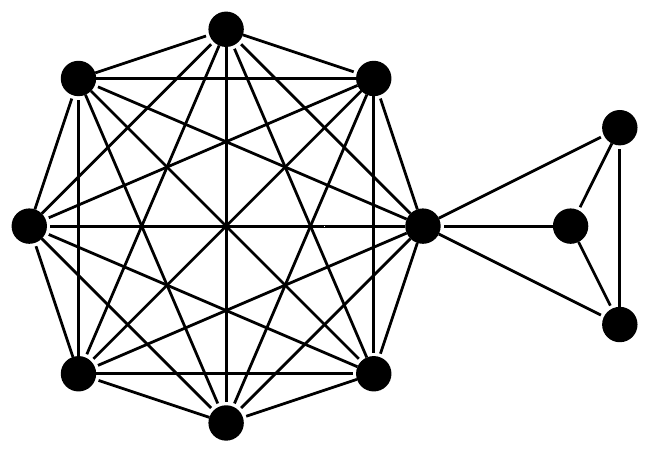}
\caption{Graphs $H_{11,3}$ (left) and $K'_{11,3}$ (right).}
\end{figure}

By construction, $H_{n,d}$ has minimum degree $d$, is nonhamiltonian, and $e(H_{n,d}) = {n-d \choose 2} + d^2 = h(n,d)$.
Elementary calculation shows that $h(n,d)> h(n,\left \lfloor \frac{n-1}{2} \right \rfloor)$ in the range
  $1\leq d\leq \left \lfloor \frac{n-1}{2} \right \rfloor$ if and only if
  $d<(n+1)/6$ and $n$ is odd or  $d<(n+4)/6$ and $n$ is even.
Hence there exists a $d_0:=d_0(n)$ such that
\begin{equation*}
   e(n,1)> e(n, 2)> \dots >e(n,d_0)=e(n, d_0+1)=\dots = e(n,\left \lfloor \frac{n-1}{2} \right \rfloor),
  \end{equation*}
  where $d_0(n):= \left \lceil \frac{n+1}{6} \right \rceil$ if $n$ is odd, and
  $d_0(n):= \left \lceil \frac{n+4}{6} \right \rceil$ if $n$ is even. Therefore $H_{n,d}$ is an extremal example of Theorem \ref{Erdos} when $d< d_0$ and $H_{n, \lfloor (n-1)/2 \rfloor}$ when $d \geq d_0$.

In~\cite{lining} and independently in~\cite{oldmain} a stability theorem for nonhamiltonian graphs with prescribed minimum degree
was proved.
Let $K'_{n,d}$ denote the edge-disjoint union of $K_{n-d}$ and $K_{d+1}$ sharing a single vertex.
An example of $K'_{11,3}$ is on the right of Fig~\ref{fig0}.

\begin{thm}[\cite{lining,oldmain}]\label{oldma}
Let $n\geq 3$ and $d\leq \left \lfloor \frac{n-1}{2} \right \rfloor$.
Suppose that $G$ is an $n$-vertex  nonhamiltonian graph  with minimum degree $\delta(G) \geq d$ such that
\begin{equation}
    e(G) > e(n,d+1)= \max\left\{h(n,d+1), h(n, \left\lfloor \frac{n-1}{2}\right\rfloor)\right\}.
\end{equation}
Then $G$ is a subgraph of either $H_{n,d}$ or $K'_{n,d}$.
\end{thm}

One of the main results of this paper shows that when $n$ is large enough with respect to $d$ and $t$,  
$H_{n,d}$ not only has the most edges among $n$-vertex nonhamiltonian graphs
with minimum degree at least $d$, but also has the most  copies of \emph{any} $t$-vertex graph.
This is an instance of a generalization of the Tur\'an problem called {\em subgraph density problem}:
 for $n \in \mathbb N$ and graphs $T$ and $H$, let $ex(n,T,H)$ denote the maximum possible number of (unlabeled)
 copies of $T$ in an $n$-vertex $H$-free graph. When $T = K_2$, we have the usual extremal number $ex(n,T,H) = ex(n, H)$.

Some notable results on the function $ex(n,T,H)$  for various combinations of $T$ and $H$ were obtained in~\cite{Erdos2,BG, alonsh, Grz, raz,FO}.
In particular, Erd\H os~\cite{Erdos2} determined $ex(n,K_s,K_t)$,  Bollob\' as and Gy\H ori~\cite{BG} found the order of magnitude
of  $ex(n,C_3,C_5)$, Alon and Shikhelman~\cite{alonsh} presented a series of bounds on $ex(n,T,H)$ for different classes of $T$ and $H$.

In this paper, we study the maximum number of  copies of $T$ in nonhamiltonian $n$-vertex graphs, i.e. $ex(n,T,C_n)$.
For two graphs $G$ and $T$, let $N(G,T)$ denote the number of {\em labeled} copies of $T$ that are subgraphs of $G$, i.e., the
number of injections $\phi: V(T)\to V(G)$ such that for each $xy\in E(T)$, $\phi(x)\phi(y)\in E(G)$.
Since for every $T$ and $H$, $|Aut (T)|\,ex(n,T,H)$ is the maximum of $N(G,T)$ over the $n$-vertex graphs $G$ not containing $H$,
some of our results are in the language of labeled copies of $T$ in $G$.
For $k \in \mathbb N$, let $N_k(G)$ denote the number of unlabeled copies of $K_k$'s in $G$. Since $|Aut (K_k)|=k!$,
 we have $N_k(G) = N(G,K_k)/k!$.

\section{Results}\label{res}

As an extension of Theorem \ref{Erdos}, we show that for each fixed graph $F$ and any $d$, if $n$ is large enough with respect to $|V(F)|$ and $d$,
then among all $n$-vertex nonhamiltonian graphs with minimum degree at least $d$, $H_{n,d}$ contains the maximum number of copies of $F$.


\begin{thm}\label{sub}For every graph $F$ with $t:=|V(F)|\geq 3$, any $d \in \mathbb N$, and any $n \geq n_0(d,t):= 4dt+3d^2 + 5t$,
if $G$ is an $n$-vertex nonhamiltonian graph with minimum degree $\delta(G) \geq d$, then $N(G,H) \leq N(H_{n,d}, F)$.
\end{thm}

On the other hand,  if  $F$ is a star $K_{1,t-1}$ and $n \leq d t - d$, then $H_{n,d}$ does not maximize $N(G, F)$.
At the end of Section~\ref{gen} we show that in this case, $N(H_{n,\lfloor (n-1)/2 \rfloor}, F) > N(H_{n,d}, F)$.
 So, the bound on  $n_0(d,t)$ in Theorem~\ref{sub} has the right order of magnitude when $d=O(t)$.

An immediate corollary of Theorem~\ref{sub} is the following generalization of Theorem~\ref{Ore}

\begin{cor}\label{subO} For every graph $F$ with $t:=|V(F)|\geq 3$ and any $n \geq n_0(t):= 9t+3$,
if $G$ is an $n$-vertex nonhamiltonian graph,  then $N(G,H) \leq N(H_{n,1}, F)$.
\end{cor}

We consider the case that $F$ is a clique in more detail. For $n,k \in \mathbb N$, define on the interval $[1, \lfloor(n-1)/2 \rfloor]$ the function
\begin{equation}
h_k(n,x) := {n - x \choose k} + x{x \choose k-1}.
\end{equation}

We use the convention that for $a \in \mathbb R$, $b \in \mathbb N$, ${a \choose b}$ is the polynomial $\frac{1}{b!} a \times (a-1) \times \ldots \times (a - b + 1)$ if $a \geq b-1$  and $0$ otherwise. 

By considering the second derivative, one can check that  for any fixed $k$ and $n$, as a function of $x$, $h_k(n,x)$ is convex on $[1, \lfloor(n-1)/2 \rfloor]$, hence it attains its maximum at one of the endpoints, $x =1$ or $x = \lfloor (n-1)/2 \rfloor$. When $k=2$, $h_2(n,x) = h(n,x)$. We prove the following generalization of Theorem \ref{Erdos}.

\begin{thm}\label{Erdos_k} Let $n, d, k$ be integers with $1 \leq d \leq \left \lfloor \frac{n-1}{2} \right \rfloor$ and $k \geq 2$.
If $G$ is a nonhamiltonian graph on $n$ vertices with minimum degree $\delta(G) \geq d$, then the number $N_k(G)$ of
$k$-cliques in $G$ satisfies
     \[N_k(G) \leq \max\left\{ h_k(n,d),h_k(n, \left \lfloor \frac{n-1}{2} \right \rfloor)\right\}\]
\end{thm}

Again, graphs $H_{n,d}$ and $H_{n,\lfloor(n-1)/2\rfloor}$ are sharpness examples for the theorem.

Finally, we present a stability version of Theorem \ref{Erdos_k}. To state the result, we first define the family of extremal graphs.

Fix $d \leq \lfloor (n-1)/2 \rfloor$. In addition to graphs $H_{n,d}$ and $K'_{n,d}$ defined above, define $H'_{n,d}$:
$V(H'_{n,d}) = A \cup B$, where $A$ induces a complete graph on $n-d-1$ vertices, $B$ is a set of $d+1$ vertices  that
induce exactly one edge, and there exists a set of vertices $\{a_1, \ldots , a_d\} \subseteq A$ such that for all $b \in B$, $N(b)-B = \{a_1, \ldots , a_d\}$.
 Note that contracting the edge in $H'_{n,d}[B]$ yields $H_{n-1,d}$. These graphs are illustrated in Fig.~\ref{fig1}

\begin{figure}[!ht]\label{fig1}
  \centering
\includegraphics[scale=.5]{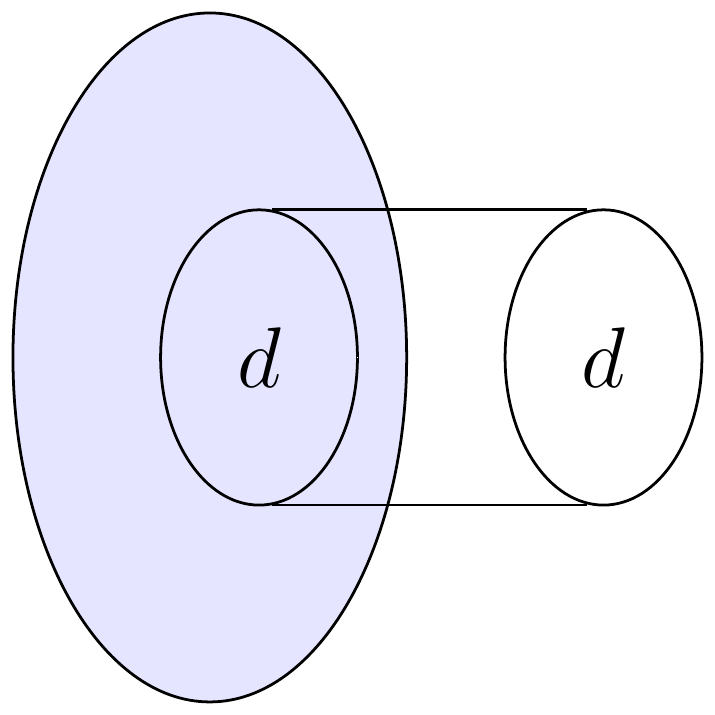}\;\;
\includegraphics[scale=.5]{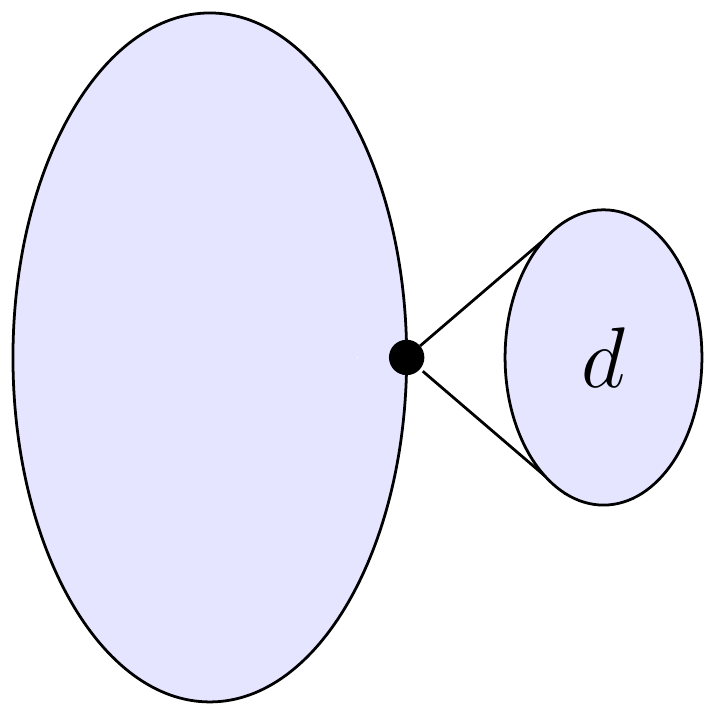}\;\;
\includegraphics[scale=.5]{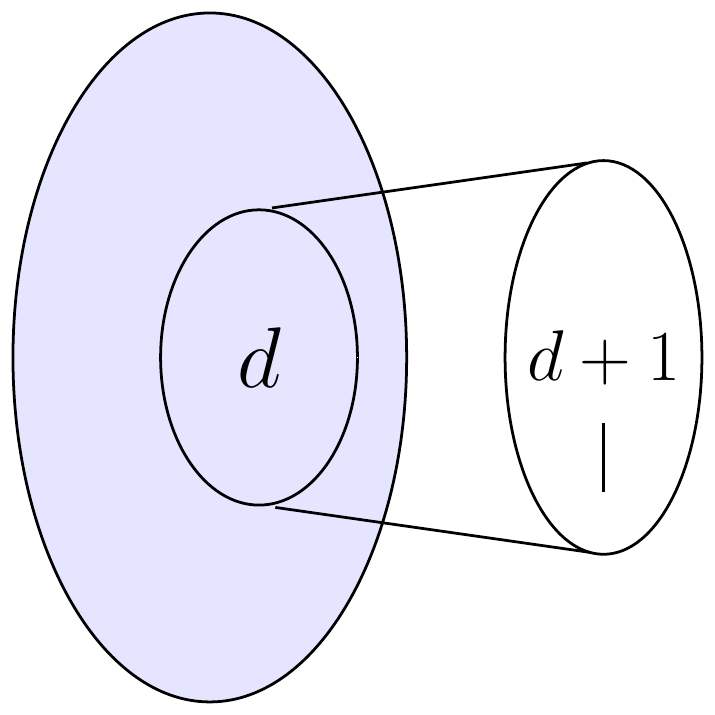}
  \caption{\footnotesize Graphs $H_{n,d}$ (left), $K'_{n,d}$ (center), and $H'_{n,d}$ (right), where shaded background indicates a complete graph.}
\end{figure}

We also have two more extremal graphs for the cases $d = 2$ or $d = 3$.
Define the nonhamiltonian  $n$-vertex
 graph $ G'_{n,2}$ with minimum degree $2$ as follows:  $V(G'_{n,2}) = A \cup B$ where $A$ induces a clique or order $n-3$, $B = \{b_1,b_2,b_3\}$ is an independent set of order $3$, and there exists $\{a_1,a_2,a_3, x\} \subseteq A$ such that $N(b_i) = \{a_i, x\}$ for $i\in \{1,2,3\}$ (see the graph on the left in Fig.~3).

 The nonhamiltonian  $n$-vertex
 graph $ F_{n,3}$ with minimum degree $3$ has  vertex set $ A \cup B$, where $A$ induces a clique of order $n-4$, $B$ induces a perfect matching on 4 vertices, and each of the vertices in $B$ is adjacent to the same two vertices in $A$ (see the graph on the right in Fig.~3).

\begin{figure}[!ht]\label{fig3}
  \centering
      \includegraphics[width=0.15\textwidth]{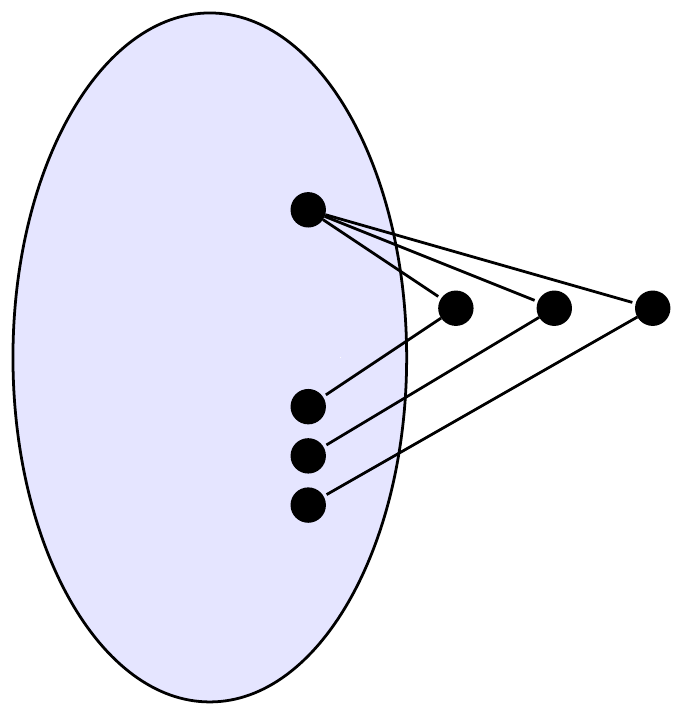}\;\;\;\;\;\;\;
      \includegraphics[width=0.15\textwidth]{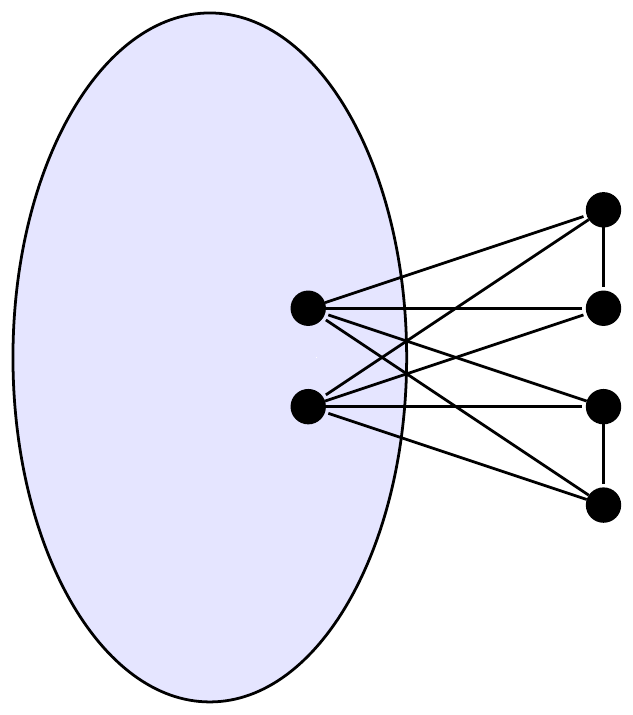}
  \caption{\footnotesize Graphs $ G'_{n,2}$ (left) and $F_{n,3}$ (right).}
\end{figure}

Our stability result is the following:

\begin{thm}\label{newma}
Let $n\geq 3$ and $1 \leq d\leq \left \lfloor \frac{n-1}{2} \right \rfloor$.
Suppose that $G$ is an $n$-vertex  nonhamiltonian graph  with minimum degree $\delta(G) \geq d$ such that there exists  $k \geq 2$ for which
\begin{equation}\label{equ2}
    N_k(G) > \max\left\{ h_k(n,d+2), h_k(n, \left\lfloor \frac{n-1}{2}\right\rfloor)\right\}.
\end{equation}
Let $\mathcal{H}_{n,d}:=\{H_{n,d}, H_{n,d+1}, K'_{n,d}, K'_{n,d+1}, H'_{n,d}\}$.
 \\
(i) If $d=2$, then  $G$ is a subgraph of  $ G'_{n,2}$ or of a graph in $\mathcal{H}_{n,2}$;\\
(ii) if $d=3$, then  $G$ is a subgraph of  $ F_{n,3}$ or of a graph in $\mathcal{H}_{n,3}$;\\
(iii) if $d=1$ or $4\leq d\leq \left \lfloor \frac{n-1}{2} \right \rfloor$, then  $G$ is a subgraph  of a graph in $\mathcal{H}_{n,d}$.
\end{thm}

The result is sharp because $H_{n,d+2}$ has $h_k(n,d+2)$ copies of $K_k$, minimum degree $d+2 > d$,  is nonhamiltonian and is not contained in
any graph in $\mathcal{H}_{n,d} \cup \{G'_{n,2}, F_{n,3}\}$.

The outline for the rest of the paper is as follows: in Section 3 we present some structural results for graphs that are edge-maximal nonhamiltonian
to be used in the proofs of the main theorems, in Section 4 we prove Theorem \ref{sub}, in Section 5 we prove Theorem \ref{Erdos_k} and give a
cliques version of Theorem \ref{oldma}, and in Section 6 we prove Theorem \ref{newma}.

\section{Structural results for saturated graphs}

We will use a classical theorem of P\'osa (usually stated as its contrapositive).

\begin{thm}[P\'osa~\cite{Posa}]\label{Posa} Let $n\geq 3$.
If $G$ is a nonhamiltonian $n$-vertex graph, then there exists  $1\leq k \leq \left \lfloor \frac{n-1}{2} \right \rfloor$
such that $G$ has a set of $k$ vertices with degree at most $k$. \end{thm}

Call a graph $G$  {\it saturated} if $G$ is nonhamiltonian but for each $uv \notin E(G)$, $G + uv$ has a hamiltonian cycle.
Ore's proof \cite{Ore} of Dirac's Theorem~\cite{Dirac} yields that
\begin{equation}\label{D1}
d(u) + d(v) \leq n-1
\end{equation}
{\em for every $n$-vertex saturated graph $G$ and for each $uv \notin E(G)$.}

We will also need two structural results for saturated graphs which are easy extensions of Lemmas 6 and 7 in \cite{oldmain}.
\begin{lem}\label{g-d} Let $G$ be a saturated $n$-vertex  graph with $N_k(G) > h_k(n, \left \lfloor \frac{n-1}{2} \right \rfloor)$ for any $k\geq2$.
Then for some $1\leq r \leq \left\lfloor \frac{n-1}{2} \right\rfloor$, $V(G)$ contains a subset $D$ of $r$ vertices of degree at most $r$
such that $G - D$ is a complete graph.
\end{lem}

\begin{proof} Since $G$ is nonhamiltonian,  by Theorem~\ref{Posa}, there exists some $1\leq r \leq \left\lfloor \frac{n-1}{2} \right\rfloor$ such that $G$ has
$r$ vertices with degree at most $r$. Pick the maximum such $r$, and let $D$ be the set of the vertices with degree at most $r$.
Since $h_k(G) > h(n, \left \lfloor \frac{n-1}{2} \right \rfloor)$, $\; r<  \left \lfloor \frac{n-1}{2} \right \rfloor$. So, by the maximality of $r$,
$|D|=r$.

Suppose there exist $x, y \in V(G) - D$ such that $xy \notin E(G)$. Among all such pairs, choose $x$ and $y$ with the maximum $d(x)$.
Since $y\notin D$, $d(y)>r$. Let $D':=V(G)-N(x)-\{x\}$ and $r':=|D'|=n-1-d(x)$.
By~\eqref{D1},
\begin{equation}\label{818}
 \mbox{ \em $d(z) \leq n - 1 - d(x)=r'\;$ for all $\;z\in D'$.
}
\end{equation}
So $D'$ is a set of $r'$ vertices of degree at most $r'$.
Since $y\in D'$, $\, r' \geq d(y) > r$. Thus by the maximality of $r$, we get
  $r' = n-1-d(x)> \left\lfloor \frac{n-1}{2}\right\rfloor$.
Equivalently, $d(x) < \lceil \frac{n-1}{2}\rceil$.
For all $z \in D'+\{x\}$, either $z \in D$ where $d(z)\leq r \leq \left\lfloor\frac{n-1}{2}\right\rfloor$, or $z \in V(G) - D$,
and so $d(z) \leq d(x)\leq \left\lfloor \frac{n-1}{2}\right\rfloor$.

Now we count the number of $k$-cliques in $G$: Among $V(G) - D'$, there are at most ${n - r' \choose k}$ $k$-cliques. Also, each vertex in $D'$ can be in at most ${r' \choose k-1}$ $k$-cliques. Therefore $N_k(G) \leq {n-r' \choose k} + r'{r' \choose k-1} \leq h_k(n, \left\lfloor \frac{n-1}{2}\right\rfloor)$, a contradiction.
\end{proof}

Also, repeating the proof of Lemma 7 in \cite{oldmain} gives the following lemma.
\begin{lem}[Lemma 7 in \cite{oldmain}]\label{hnd}  Under the conditions of Lemma \ref{g-d}, if $r = \delta(G)$, then $G = H_{n,\delta(G)}$ or $G = K'_{n,\delta(G)}$.
\end{lem}

\section{Maximizing the number of copies of a given graph and a proof of Theorem \ref{sub}}\label{gen}

In order to prove Theorem \ref{sub}, we first show that for any fixed graph $F$ and any $d$, of the two extremal graphs of Lemma \ref{hnd},
if $n$ is large then $H_{n,d}$ has at least as many copies of $F$ as $K'_{n,d}$.

\begin{lem}\label{subk}For any $d,t,n \in \mathbb N$ with $n \geq 2 dt + d + t$ and any graph $F$ with $t=|V(F)|$
we have $N(K'_{n,d}, F) \leq N(H_{n,d}, F)$.
\end{lem}


\begin{proof}
Fix $F$ and $t = |V(F)|$. Let $K'_{n,d} = A \cup B$ where $A$ and $B$ are cliques of order $n-d$ and $d+1$ respectively
and $A \cap B = \{v^*\}$, the cut vertex of $K'_{n,d}$. Also, let $D$ denote the independent set of order $d$ in $H_{n,d}$.
We may assume $d \geq 2$, because $H_{n,1} = K'_{n,1}$.
If $x$ is an isolated vertex of $F$ then for any $n$-vertex graph $G$ we have
  $N(G,F)= (n-t+1) N(G, F-x)$.
So it is enough to prove the case $\delta (F)\geq 1$, and we may also assume $t\geq 3$.

Because both $K'_{n,d}[A]$ and $H_{n,d}- D$ are cliques of order $n-d$, the number of embeddings of $F$ into $K'_{n,d}[A]$ is the same as the number of embeddings of $F$ into $H_{n,d} - D$. So it remains to compare only the number of embeddings in $\Phi:= \{ \varphi: V(F)\to V(K'_{n,d})$ such that $\varphi(F)$ intersects $B - v^* \}$
to the number of embeddings in $\Psi:= \{ \psi: V(F)\to V(H_{n,d})$ such that $\psi(F)$ intersects $D\}$.

Let $C\cup \overline{C}$ be a partition of the vertex set $V(F)$, $s:=|C|$.
Define the following classes of $\Phi$ and $\Psi$\newline
${}$\quad ---\quad $\Phi(C):= \{ \varphi: V(F)\to V(K'_{n,d})$ such that $\varphi(C)$ intersects $B- v^*$, $\varphi(C)\subseteq B$,
  and $\varphi(\overline{C})\subseteq V -B\}$,
\newline
${}$\quad ---\quad $\Psi(C):= \{ \psi: V(F)\to V(H_{n,d})$ such that  $\psi(C)$ intersects $D$, $\psi(C)\subseteq (D\cup N(D))$,
  and $\psi(\overline{C})\subseteq V- (D\cup N(D))\}$.
\newline
By these definitions, if $C\neq C'$ then $\Phi(C)\cap \Phi(C')= \emptyset$,  and  $\Psi(C)\cap \Psi(C')= \emptyset$.
Also $\bigcup_{\emptyset \neq C \subseteq V(F)} \Phi(C)= \Phi$.
We claim that for every $C\neq \emptyset$, 
\begin{equation}\label{eqC}
|\Phi(C)|\leq |\Psi(C)|.
   \end{equation}
Summing up the number of embeddings over all choices for $C$ will prove the lemma.
If $\Phi(C)=\emptyset$, then~\eqref{eqC} obviously holds. So from now on, we consider the cases when $\Phi(C)$ is not empty, implying
 $1\leq s\leq d+1$.

{\bf Case 1}: There is an $F$-edge joining $\overline{C}$ and $C$.
So there is a vertex $v\in C$ with $N_F(v)\cap \overline{C}\neq \emptyset$.
Then for every mapping  $\varphi\in \Phi(C)$, the vertex $v$ must be mapped to $v^*$ in $K'_{n,d}$, $\varphi(v)= v^*$.
So this vertex $v$ is uniquely determined by $C$.
Also, $\varphi(C)\cap (B- v^*) \neq \emptyset$ implies $s\geq 2$.
The rest of $C$ can be mapped arbitrarily to $B-v^*$ and $\overline {C}$ can be mapped arbitrarily to $A-v^*$. We obtained that
   $|\Phi(C)|= (d)_{s-1}(n-d-1)_{t-s}$.

We make a lower bound for $|\Psi(C)|$ as follows.
We define a $\psi\in \Psi(C)$ by the following procedure.
Let $\psi(v)=x\in N(D)$ (there are $d$ possibilities), then map some vertex of $C-v$ to a vertex $y\in D$ (there are $(s-1)d$ possibilities).
Since $N+y$ forms a clique of order $d+1$ we may embed the rest of $C$ into $N-v$ in $(d-1)_{s-2}$ ways and finish
 embedding of $F$ into $H_{n,d}$ by arbitrarily placing the vertices of $\overline{C}$ to $V-(D\cup N(D))$.
We obtained that   $|\Psi(C)|\geq d^2(s-1)(d-1)_{s-2}(n-2d)_{t-s} = d(s-1)(d)_{s-1}(n-2d)_{t-s}$.

  Since $s\geq 2$ we have that
\begin{eqnarray*}
\frac{|\Psi(C)|}{|\Phi(C)|}\geq \frac{d(s-1)(d)_{s-1}(n-2d)_{t-s}}{(d)_{s-1}(n-d-1)_{t-s}}
       &\geq & d(2-1) \left(\frac{n-2d+1-t+s}{n-d-t+s}\right)^{t-s} \\
 & = &  d\left( 1 - \frac{d-1}{n-d-t+s}\right)^{t-s} \\
 & \geq &  d\left( 1 - \frac{(d-1)(t-s)}{n-d-t+s}\right)\\
 & \geq &  d\left(1 - \frac{(d-1)t}{n-d-t} \right)\\
 &> & 1\text{ when }  n> dt + d + t.
 \end{eqnarray*}

{\bf Case 2}:
$C$ and $\overline C$ are not connected in $F$.
We may assume $s \geq 2$ since $C$ is a union of components with $\delta(F)\geq 1$.
In $K'_{n,d}$ there are at exactly $(d+1)_{s}(n-d-1)_{t-s}$ ways to embed $F$ into $B$ so that only $C$ is mapped into $B$ and
 $\overline{C}$ goes to $A-v^*$, i.e.,  $|\Phi(C)|= (d+1)_{s}(n-d-1)_{t-s}$.

We make a lower bound for $|\Psi(C)|$ as follows.
We define a $\psi\in \Psi(C)$ by the following procedure.
Select any vertex $v\in C$ and map it to some vertex in $D$ (there are $sd$ possibilities), then map  $C-v$ into  $N(D)$ (there are $(d)_{s-1}$ possibilities)  and finish
 embedding of $F$ into $H_{n,d}$ by arbitrarily placing the vertices of $\overline{C}$ to $V-(D\cup N(D))$.
We obtained that   $|\Psi(C)|\geq ds(d)_{s-1}(n-2d)_{t-s}$.
 We have
\begin{eqnarray*}
\frac{|\Psi(C)|}{|\Phi(C)|}\geq
\frac{ds(d)_{s-1}(n-2d)_{t-s}}{(d+1)_{s}(n-d-1)_{t-s}} & \geq & \frac{ds}{d+1} \left(1 - \frac{(d-1)t}{n-d-t} \right)\\
&\geq & \frac{2d}{d+1}\left(1 - \frac{(d-1)t}{n-d-t} \right) \text{ because } s \geq 2 \\
& > & 1 \text{ when } n >2dt + d + t.
\end{eqnarray*}
\end{proof}

We are now ready to prove Theorem ~\ref{sub}.

{\bf Theorem \ref{sub}}. {\em For every graph $F$ with $t:=|V(F)|\geq 3$, any $d \in \mathbb N$, and any $n \geq n_0(d,t):= 4dt+3d^2 + 5t$,
if $G$ is an $n$-vertex nonhamiltonian graph with minimum degree $\delta(G) \geq d$, then $N(G,H) \leq N(H_{n,d}, F)$.}


\begin{proof}
Let $d \geq 1$. Fix a graph $F$ with $|V(F)| \geq 3$ (if $|V(F)| = 2$, then either $F=K_2$ or $F=\overline K_2$). The case where $G$ has isolated vertices can be handled by induction on the number of isolated vertices, hence we may assume each vertex has degree at least 1. Set
\begin{equation}\label{alli}
n_0= 4dt + 3d^2 + 5t.
\end{equation}

 Fix a nonhamiltonian graph $G$ with $|V(G)|=n \geq n_0$ and $\delta(G) \geq d$
 such that $N(G,F) > N(H_{n,d}, F) \geq (n-d)_t$. We may assume that $G$ is saturated, as the number of copies of $F$ can only increase when we add edges to $G$.

Because $n \geq 4dt + t$ by (\ref{alli}),
\begin{eqnarray*}
\frac{(n-d)_t }{(n)_t} & \geq & \left(\frac{n-d-t}{n-t}\right)^t
 =  \left(1 - \frac{d}{n-t}\right)^t\\
& \geq & 1 - \frac{dt}{n-t}
 \geq  1-\frac{1}{4} = \frac{3}{4}.
\end{eqnarray*}
So, $(n-d)_t \geq \frac{3}{4} (n)_t$.

After mapping edge $xy$ of $F$ to an edge of $G$ (in two labeled ways), we obtain the loose upper bound, \[
2e(G)(n-2)_{t-2} \geq N(G,F) \geq (n-d)_t \geq \frac{3}{4} (n)_t,\]therefore
\begin{equation} \label{eps}
 e(G)   \geq \frac{3}{4}{n \choose2} > h_2(n,\lfloor (n-1)/2 \rfloor).
\end{equation}
By P\'osa's theorem (Theorem \ref{Posa}), there exists some $d \leq r \leq \lfloor (n-1)/2 \rfloor$ such that $G$ contains
a set $R$ or $r$ vertices with degree at most $r$. Furthermore by (\ref{eps}), $r < d_0$. So by integrality, $r\leq d_0-1 \leq (n+3)/6$.
 If $r = d$, then by Lemma \ref{hnd}, either $G = H_{n,d}$ or $G=K'_{n,d}$. By Lemma ~\ref{subk} and (\ref{alli}), $G = H_{n,d}$, a contradiction.
So we have $r \geq d+1$.

Let $\mathcal{I}$ denote the family of all nonempty independent sets in $F$.
For $I\in \mathcal{I}$, let $i=i(I):= |I|$ and $j=j(I):=|N_F(I)|$.
Since $F$ has no isolated vertices, $j(I)\geq 1$ and so $i \leq t-1$ for each $I\in \mathcal{I}$. Let $\Phi(I)$ denote the set of
embeddings
  $\varphi: V(F)\to V(G)$ such that $\phi(I)\subseteq R$ and $I$ is a maximum independent
subset of  $\phi^{-1}(R\cap \varphi(F))$.
Note that  $\varphi(I)$ is not necessarily independent in $G$. We show that
\begin{equation} \label{0322}
|\Phi(I)|\leq (r)_ir(n-r)_{t-i-1}.
\end{equation}
Indeed, there are $(r)_i$ ways to choose $\phi(I)\subseteq R$. After that, since each vertex in $R$  has at most $r$ neighbors in $G$,
there are at most $r^j$ ways to embed $N_F(I)$ into $G$. By the maximality of $I$, all vertices of $F-I-N_F(I)$ should be mapped to
$V(G)-R$. There are at most $(n-r)_{t-i-j}$ to do it. Hence $|\Phi(I)|\leq (r)_ir^j(n-r)_{t-i-j}$. Since $2r+t\leq 2(d_0-1)+t<n$,
this implies~\eqref{0322}.

Since each  $\varphi: V(F)\to V(G)$ with $\varphi(V(F))\cap R\neq \emptyset$ belongs to $\Phi(I)$ for some nonempty  $I\in \mathcal{I}$,
\eqref{0322}~implies
\begin{equation} \label{03222}
N(G,F)\leq (n-r)_t+\sum_{\emptyset\neq I\in \mathcal{I} }|\Phi(I)|\leq
(n-r)_t+\sum_{i=1}^{t-1}\binom{t}{i}(r)_ir(n-r)_{t-i-1}.
\end{equation}

Hence
\begin{eqnarray*}
\frac{N(G,F)}{N(H_{n,d},F)} &\leq & \frac{(n-r)_t+\sum_{i=1}^{t-1}\binom{t}{i}(r)_ir(n-r)_{t-i-1}}
{(n-d)_t}\\
& \leq & \frac{(n-r)_t}{(n-d)_t} + \frac{1}{(n-d)_t}\times \frac{r}{n-r-t+2}\sum_{i=1}^{t-1}\binom{t}{i}(r)_i (n-r)_{t-i}
\\
&  =& \frac{(n-r)_t}{(n-d)_t} + \frac{(n)_t -(n-r)_t -(r)_t}{(n-d)_t}\times \frac{r}{n-r-t+2}
\\
& \leq & \frac{(n-r)_t}{(n-d)_t} \times \frac{n-t+2 -2r}{n-t+2-r} + \frac{(n)_t }{(n-d)_t}\times \frac{r}{n-t+2-r}:= f(r).
\end{eqnarray*}

Given fixed $n,d,t$, we claim that the real function $f(r)$ is convex for $0< r< (n-t+2)/2$.


Indeed, the first term $g(r) := \frac{(n-r)_t}{(n-d)_t} \times \frac{n-t+2 -2r}{n-t+2-r}$ is a product of $t$ linear terms in each of which $r$ has a negative coefficient (note that the $n-t+2-r$ term cancels out with a factor of $n-r-t+2$ in $(n-r)_t$). Applying product rule, the first derivative $g'$ is a sum of $t$ products, each with $t-1$ linear terms. For $r < (n-t+2)/2$, each of these products is negative, thus $g'(r) < 0$. Finally, applying product rule again, $g''$ is the sum of $t(t-1)$ products. For $r < (n-t+2)/2$ each of the products is positive, thus $g''(r) > 0$.

Similarly, the second factor of the second term (as a real function of $r$, of the form  $r/(c-r)$) is convex for $r< n-t+2$.

We conclude that in the interval $[d+1, (n+3)/6]$ the function  $f(r)$ takes its maximum either at one of the endpoints $r = d+1$ or $r = (n+3)/6$. We claim that $f(r)<1$ at both end points.

In case of $r=d+1$ the first factor of the first term equals $(n-d-t)/(n-d)$.
To get an upper bound for the first factor of the second term one can use the inequality
  $\prod (1+x_i) < 1+ 2 \sum x_i $ which holds for any number of non-negative $x_i$'s if $0< \sum x_i \leq 1$.
Because $dt/(n-d-t+1)\leq 1$ by (\ref{alli}), we obtain that
\begin{eqnarray*}
 f(d+1) & < &  \frac{n-d-t}{n-d} \times \frac{n-t -2d}{n-t-d+1} + \left(1+ \frac{2dt}{n-d-t+1}\right) \times \frac{d+1}{n-t-d+1}\\
& = & \left(1-\frac{t}{n-d} \right) \times \left(1-\frac{d+1}{n-t-d+1} \right) + \left(\frac{d+1}{n-t-d+1}\right) + \left(\frac{2dt(d+1)}{(n-t-d+1)^2} \right) \\
& = & 1 - \frac{t}{n-d} + \frac{t}{n-d}  \times \frac{d+1}{n-t-d+1} + \frac{t}{n-d} \times \frac{2d(d+1)}{n-t-d+1} \times  \frac{n-d}{n-t-d+1} \\
& = & 1 - \frac{t}{n-d} \times \left(1 - \frac{d+1}{n-t-d+1} - \frac{2d(d+1)}{n-t-d+1} \times \left(1 + \frac{t-1}{n-t-d+1}\right)\right)
\\
&<&1 - \frac{t}{n-d} \times (1 - \frac{1}{4t}  - \frac{2}{3}(1 + \frac{1}{4d}))\\
&\leq & 1 - \frac{t}{n-d}\times( 1 - 1/12 - 2/3\times 5/4)\\
&<& 1.
\end{eqnarray*}
Here we used that $n\geq 3d^2+ 2d +t$ and $n \geq 4dt + 5t + d$ by (\ref{alli}), $t\geq 3$, and $d\geq 1$.

To bound $f(r)$ for other values of $r$, let us use $1+x\leq e^x$ (true for all $x$).  We get
$$
  f(r) <  \exp\left\{ - \frac{(r-d)t}{n-d-t+1}  \right\} + \frac{r}{n-r-t+2}\times \exp\left\{
\frac{dt}{n-d-t+1}  \right\}.
$$
When $r=(n+3)/6$, $t\geq 3$, and $n\geq 24d$ by (\ref{alli}),  the first term is at most $e^{-18/46}=0.676...$.
Moreover, for $n\geq 9t$ (\ref{alli}) (therefore $n \geq 27$) we get that $\frac{r}{n-r-t+2}$ is maximized when $t$ is maximized, i.e., when $t  = n/9$.  The whole term is at most $(3n+9)/(13n+27) \times e^{1/4} \leq 5/21 \times e^{1/4}=0.305...$,
 so in this range, $f((n+3)/6)<1$.

 By the convexity of $f(r)$, we have $N(G,F) < N(H_{n,d}, F)$.
\end{proof}

When $F$ is a star, then it is easy to determine $\max N(G,F)$ for all $n$.

\begin{claim}\label{cl:star} 
Suppose $F=K_{1,t-1}$ with $t:=|V(F)|\geq 3$, and $t\leq n$ and $d$ are integers with $1\leq d \leq \lfloor (n-1)/2\rfloor$.
If $G$ is an $n$-vertex nonhamiltonian graph with minimum degree $\delta(G) \geq d$, then 
\begin{equation}\label{eq:star}
  N(G,F) \leq \max \left\{ H_{n,d},  H_{n,\lfloor (n-1)/2 \rfloor}\right\},  
    \end{equation}
and equality holds if and only if  $G\in \left\{ H_{n,d},  H_{n,\lfloor (n-1)/2 \rfloor}\right\}$. 
\end{claim}

\begin{proof}
The number of copies of stars in a graph $G$ depends only on the degree sequence of the graph: 
if a vertex $v$ of a graph $G$ has degree $d(v)$, then there are $(d(v))_{t-1}$ labeled copies of $F$ in $G$ where $v$ is the center vertex. We have 
\begin{equation}\label{eq:starformula}
  N(G,F)= \sum_{v\in V(G)} \binom{d(v)}{t-1}.
  \end{equation} 
Since $G$ is nonhamiltonian, P\'osa's theorem yields an $r\leq \lfloor (n-1)/2 \rfloor$, and an $r$-set $R\subset V(G)$ such that 
  $d_G(v)\leq r$ for all $v\in R$.
Take the minimum such $r$, then there exists a vertex $v\in R$ with $\deg(v)=r$. 
We may also suppose that $G$ is edge-maximal nonhamiltonian, so Ore's condition~\eqref{D1} holds. 
It implies that $\deg(w)\leq n-r-1$ for all $w\notin N(v)$.
Altogether we obtain that $G$ has $r$ vertices of degree at most $r$, at least $n-2r$ vertices (those in $V(G)-R-N(v)$) of degree at most $(n-r-1)$. This implies that
  the right hand side of~\eqref{eq:starformula} is at most
  \[
   r\times (r)_{t-1}+ (n-2r)\times (n-r-1)_{t-1} + r\times (n-1)_{t-1}= N(H_{n,r}, F).
  \]
(Here equality holds only if $G=H_{n,r}$). Note that $r \in [d, \lfloor \frac{1}{2}(n-1) \rfloor]$.
Since for given $n$ and $t$ the function $N(H_{n,r}, F)$ is strictly convex in $r$, it takes its maximum at one of the endpoints of the interval. 
   \end{proof}

\begin{rem}\label{rem13} 
As it was mentioned in Section~\ref{res}, $O(dt)$ is the right order for $n_0(d,t)$ when $d=O(t)$.
    \end{rem}
  
  To see this, fix $d \in \mathbb N$ and let $F$ be the star on $t\geq 3$ vertices.
If $d < \lfloor (n-1)/2 \rfloor$, $t\leq n$ and $n\leq dt-d$, then $H_{n,\lfloor(n-1)/2\rfloor}$ contains more copies of $F$ than $H_{n,d}$ does, 
the maximum in~\eqref{eq:star} is reached for $r=\lfloor(n-1)/2\rfloor$. 
We present the calculation below only for $2d+7\leq n\leq dt-d$, the case
 $2d+3\leq n\leq 2d+6$ can be checked by hand by plugging $n$ into the first line of the formula below.
We can proceed as follows.

\begin{eqnarray*}
N(H_{n,\lfloor (n-1)/2 \rfloor}, F) - N(H_{n,d}, F) &= & 
  \Big(\lfloor (n-1)/2 \rfloor (n-1)_{t-1} + \lceil (n+1)/2 \rceil (\lfloor(n-1)/2 \rfloor)_{t-1}\Big)\\
  && - \Big(d(n-1)_{t-1} + (n-2d)(n-d-1)_{t-1} + d(d)_{t-1}\Big)
\\
&= &\Big(\lfloor(n-1)/2 \rfloor - d \Big) (n-1)_{t-1} -(n-2d)(n-d-1)_{t-1}\\
&&+ \lceil (n+1)/2 \rceil (\lfloor(n-1)/2 \rfloor)_{t-1} - d(d)_{t-1}\\
& > & \Big(\lfloor(n-1)/2 \rfloor - d \Big) (n-1)_{t-1} - \Big((n-2d)(1-d/n)^{t-1}\Big) (n-1)_{t-1}\\
&>& (n-1)_{t-1} \left( \lfloor (n-1)/2 \rfloor - d - (n-2d)e^{-(dt-d)/n}\right) \\
&\geq & (n-1)_{t-1} \left( \lfloor (n-1)/2 \rfloor - d - (n-2d)/e \right)\\
&\geq &0.
\end{eqnarray*}

\section{ Theorem \ref{Erdos_k} and a stability version of it}
In general, it is difficult to calculate the exact value of $N(H_{n,d}, F)$ for a fixed graph $F$.
However, when $F = K_k$, we have $N(H_{n,d}, K_k) = h_k(n,d)k!$. Recall Theorem \ref{Erdos_k}:

 {\em Let $n, d, k$ be integers with $1 \leq d \leq \left \lfloor \frac{n-1}{2} \right \rfloor$ and $k \geq 2$.
If $G$ is a nonhamiltonian graph on $n$ vertices with minimum degree $\delta(G) \geq d$, then
     \[N_k(G) \leq \max\left\{ h_k(n,d),h_k(n, \left \lfloor \frac{n-1}{2} \right \rfloor)\right\}.\]}

{\em Proof of Theorem \ref{Erdos_k}}.
By Theorem~\ref{Posa}, because $G$ is nonhamiltonian, there exists an $r\geq d$ such that $G$ has $r$ vertices of degree at most $r$.
Denote this set of vertices by $D$. Then $N_k(G-D) \leq {n-r \choose k}$, and every vertex in $D$ is contained in at most ${r \choose k-1}$ copies of $K_k$.
Hence $N_k(G) \leq h_k(n,r)$. The theorem follows from the convexity of $h_k(n,x)$. \qed

Our older stability theorem (Theorem \ref{oldma}) also translates  into the the language of cliques, giving a stability theorem for Theorem \ref{Erdos_k}:

\begin{thm}\label{oldma2}
Let $n\geq 3$, and $d\leq \left \lfloor \frac{n-1}{2} \right \rfloor$.
Suppose that $G$ is an $n$-vertex  nonhamiltonian graph  with minimum degree $\delta(G) \geq d$ and there exists a $k \geq 2$ such that
\begin{equation}
    N_k(G) > \max\left\{ h_k(n,d+1), h_k(n, \left\lfloor \frac{n-1}{2}\right\rfloor)\right\}.
\end{equation}
Then $G$ is a subgraph of either $H_{n,d}$ or $K'_{n,d}$.
\end{thm}
\begin{proof}
Take an edge-maximum counterexample $G$ (so we may assume $G$ is saturated). By Lemma \ref{g-d}, $G$ has a set
$D$ of $r \leq \lfloor (n-1)/2 \rfloor$ vertices such that $G-D$ is a complete graph.
If $r \geq d+1$, then $N_k(G) \leq \max\left\{ h_k(n,d+1), h_k(n, \left\lfloor \frac{n-1}{2}\right\rfloor)\right\}$. Thus $r = d$, and we may apply Lemma \ref{hnd}.
\end{proof}

\section{Discussion and  proof of Theorem \ref{newma}}

One can try to refine Theorem \ref{oldma} in the following direction: What happens when we consider $n$-vertex nonhamiltonian graphs with
minimum degree at least $d$ and less than $e(n,d+1)$ but
more than $e(n,d+2)$ edges?

Note that for $d< d_0(n)-2$, \[e(n,d) - e(n,d+2) =  2n - 6d - 7,
\]
 which is greater than $n$. Theorem \ref{newma} answers the question above in a more general form---in terms of $s$-cliques instead of edges. In other words,
we classify all $n$-vertex nonhamiltonian graphs with more than $\max\left\{ h_s(n,d+2), h_s(n, \left\lfloor \frac{n-1}{2}\right\rfloor)\right\}$ copies of $K_s$.

 As in Lemma \ref{oldma2}, such $G$ can be a subgraph of $H_{n,d}$ or $K'_{n,d}$. Also, $G$ can be a subgraph of $H_{n,d+1}$ or $K'_{n,d+1}$.
Recall the graphs $H_{n,d}, K'_{n,d}, H'_{n,d}, G'_{n,2},$ and $F_{n,3}$ defined in the first two sections of this paper and the statement of Theorem \ref{oldma}:

\begin{figure}[!ht]\label{fig7}
  \centering
\includegraphics[scale=.4]{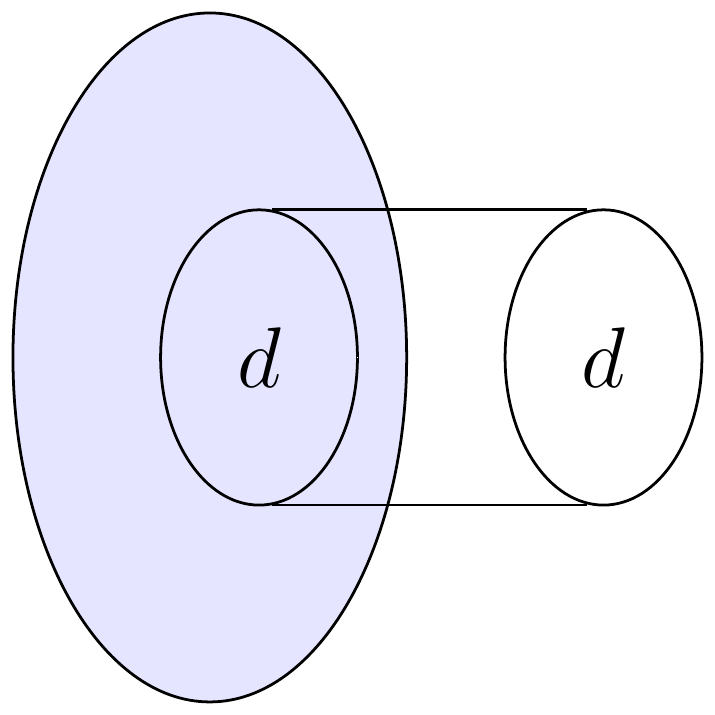}\;\;
\includegraphics[scale=.4]{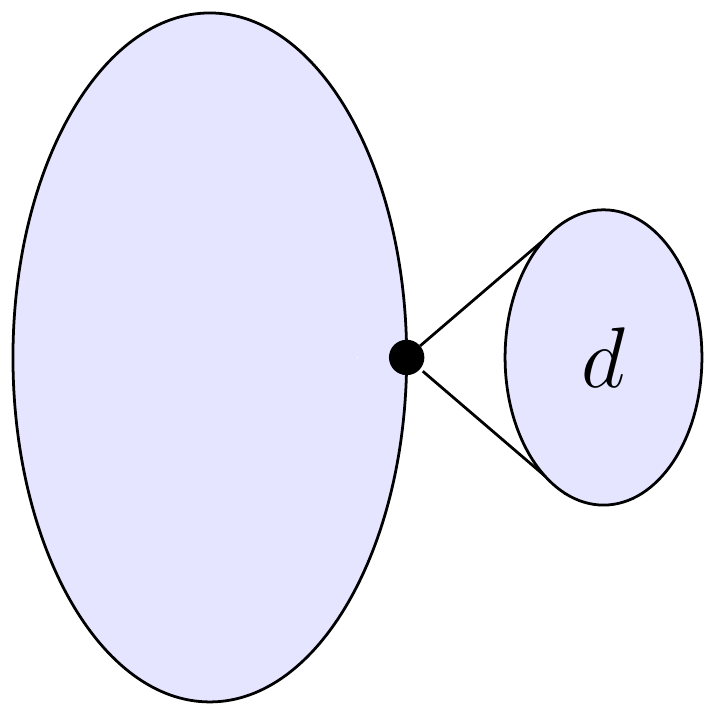}\;\;
\includegraphics[scale=.4]{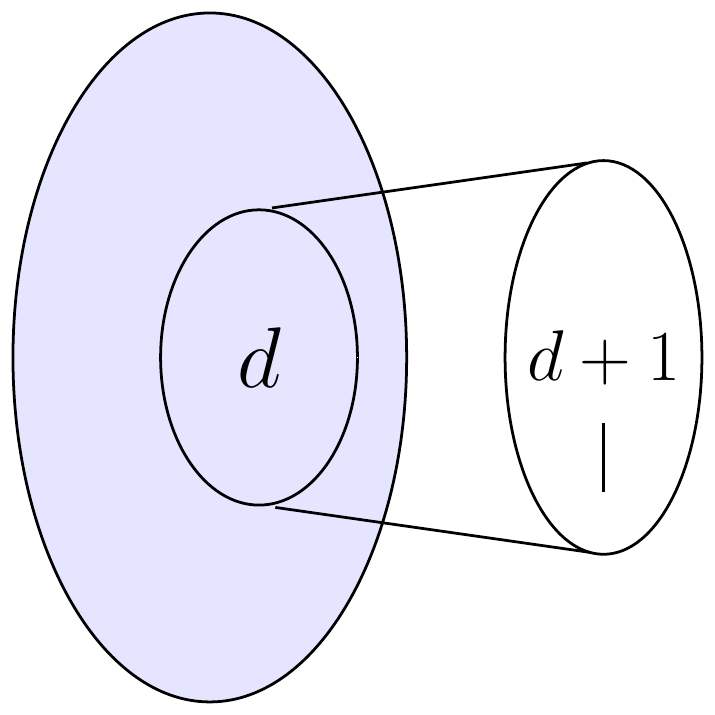}
      \includegraphics[scale=.4]{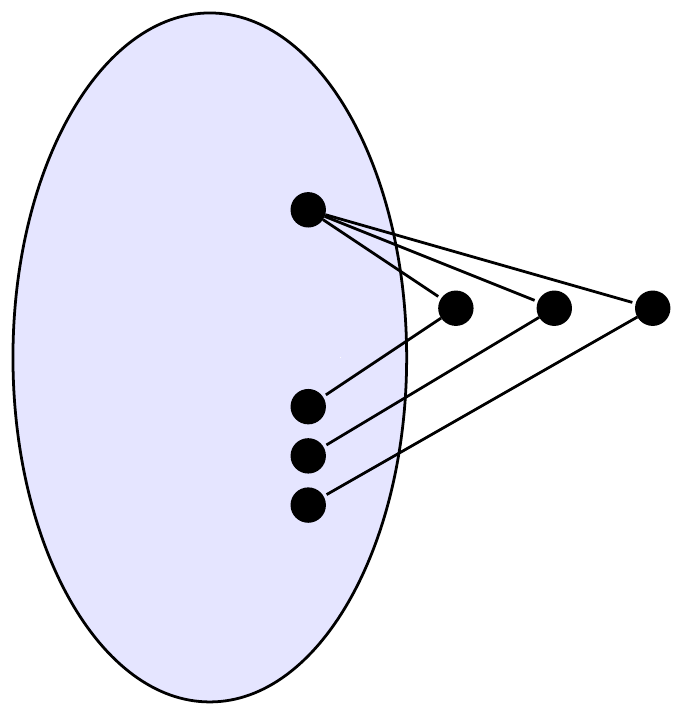}\;\;\;\;\;\;\;
      \includegraphics[scale=.4]{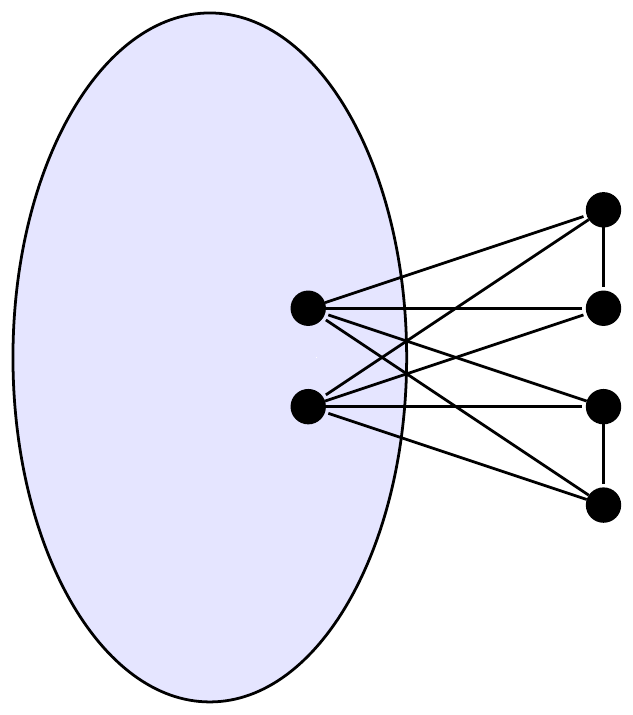}

  \caption{\footnotesize Graphs $H_{n,d}, K'_{n,d}, H'_{n,d}, G'_{n,2}, $ and $F_{n,3}$.}
\end{figure}

{\bf Theorem \ref{newma}}.
{\em Let $n\geq 3$ and $1 \leq d\leq \left \lfloor \frac{n-1}{2} \right \rfloor$.
Suppose that $G$ is an $n$-vertex  nonhamiltonian graph  with minimum degree $\delta(G) \geq d$ such that exists a $k \geq 2$ for which
 $$   N_k(G) > \max\left\{ h_k(n,d+2), h_k(n, \left\lfloor \frac{n-1}{2}\right\rfloor)\right\}.$$
Let $\mathcal{H}_{n,d}:=\{H_{n,d}, H_{n,d+1}, K'_{n,d}, K'_{n,d+1}, H'_{n,d}\}$.

(i) If $d=2$, then  $G$ is a subgraph of  $ G'_{n,2}$ or of a graph in $\mathcal{H}_{n,2}$;\\
(ii) if $d=3$, then  $G$ is a subgraph of  $ F_{n,3}$ or of a graph in $\mathcal{H}_{n,3}$;\\
(iii) if $d=1$ or $4\leq d\leq \left \lfloor \frac{n-1}{2} \right \rfloor$, then  $G$ is a subgraph  of a graph in $\mathcal{H}_{n,d}$.
}

\begin{proof}
Suppose $G$ is a counterexample to Theorem \ref{newma} with the most edges. Then $G$ is saturated. In particular,
 degree condition~\eqref{D1} holds for $G$. So by Lemma \ref{g-d}, there exists
an $d\leq r \leq \lfloor (n-1)/2 \rfloor$ such that $V(G)$ contains a subset $D$ of $r$ vertices of degree at most $r$ and $G - D$ is a complete graph.

If $r \geq d+2$, then because $h_k(n,x)$ is convex, $N_k(G) \leq h_k(n,r) \leq \max\left\{ h_k(n,d+2), h_k(n, \left\lfloor \frac{n-1}{2}\right\rfloor)\right\}$. Therefore either $r =d$ or $r = d+1$.  In the case that $r =d$ (and so $r = \delta(G)$), Lemma \ref{hnd} implies that $G \subseteq H_{n,d}$. So we may assume that $r = d+1$.

If $\delta(G)\geq d+1$, then we simply apply Theorem~\ref{oldma} with $d+1$ in place of $d$ and get $G \subseteq H_{n,d+1}$ or $G \subseteq K'_{n,d+1}$. So, from now on we may assume
\begin{equation}\label{s21}
\delta(G)=d.
\end{equation}

Now~\eqref{s21} implies that our theorem holds for $d=1$, since each graph with minimum degree exactly $1$ is a subgraph of $H_{n,1}$. So, below
$2 \leq d\leq \left \lfloor \frac{n-1}{2} \right \rfloor$.


Let $N := N(D) - D \subseteq V(G) - D$. The next claim will be used many times throughout the proof.

\begin{lem}\label{le0}
(a) If there exists a vertex $v \in D$ such that $d(v) = d+1$, then $N(v) - D = N$.
\\
(b) If there exists a vertex $u \in N$ such that $u$ has at least 2 neighbors in $D$, then $u$ is adjacent to all vertices in $D$.

\end{lem}
\begin{proof}
If  $v \in D$, $d(v) = d+1$ and some  $u \in N$ is not adjacent to $v$, then $d(v) + d(u) \geq d+1 + (n-d-2) + 1 = n$. A
contradiction to~\eqref{D1} proves (a).

Similarly, if  $u \in N$  has at least 2 neighbors in $D$ but is not adjacent to some $v\in D$, then
 $d(v) + d(u) \geq d + (n-d-2) + 2 = n$, again contradicting~\eqref{D1}.
\end{proof}

Define $S := \{u \in V(G) - D: u \in N(v)\text{ for all } v \in D\}$, $s := s$, and $S':= V(G) - D - S$. By Lemma~\ref{le0}~(b),
each vertex in $S'$ has at most one neighbor in $D$. So,  for each $v\in D$, call the neighbors of $v$ in $S'$ {\em the private neighbors of $v$}.

\medskip
We claim that
\begin{equation}\label{ni}
\mbox{\em $D$ is not  independent. }
\end{equation}
Indeed, assume $D$ is  independent.
If there exists a vertex $v \in D$ with $d(v) = d+1$, then by Lemma~\ref{le0}~(b), $N(v) - D = N$. So, because $D$ is  independent, $G \subseteq H_{n,d+1}$.
Assume now that every vertex $v \in D$ has degree $d$, and let $D = \{v_1, \ldots , v_{d+1}\}$.

If $s\geq d$, then because each $v_i\in D$ has degree $d$, $s = d$ and $N = S$. Then $G \subseteq H_{n,d+1}$. If $s \leq d-2$, then each vertex $v_i \in D$ has at least two private neighbors in $S'$; call these private neighbors $x_{v_i}$ and $y_{v_i}$.  The path $x_{v_1} v_1 y_{v_1} x_{v_2} v_2 y_{v_2}  \ldots  x_{v_{d+1}} v_{d+1} y_{v_{d+1}}$ contains all vertices in $D$ and can be extended to a hamiltonian cycle of $G$, a contradiction.

Finally, suppose $s = d-1$. Then every vertex $v_i \in D$ has exactly one private neighbor.
 Therefore $G = G'_{n,d}$ where $G'_{n,d}$ is composed of a clique $A$ of order $n-d-1$ and an independent set $D = \{v_1, \ldots, v_{d+1}\}$,
and there exists a set $S \subset A$ of size $d-1$ and distinct vertices $z_1, \ldots, z_{d+1}$ such that
for $1 \leq i \leq d+1$, $N(v_i) = S \cup z_i$. Graph $G'_{n,d}$ is illustrated in Fig.~\ref{nf1}.

\begin{figure}[!ht]\label{nf1}
  \centering
    \includegraphics[width=0.25\textwidth]{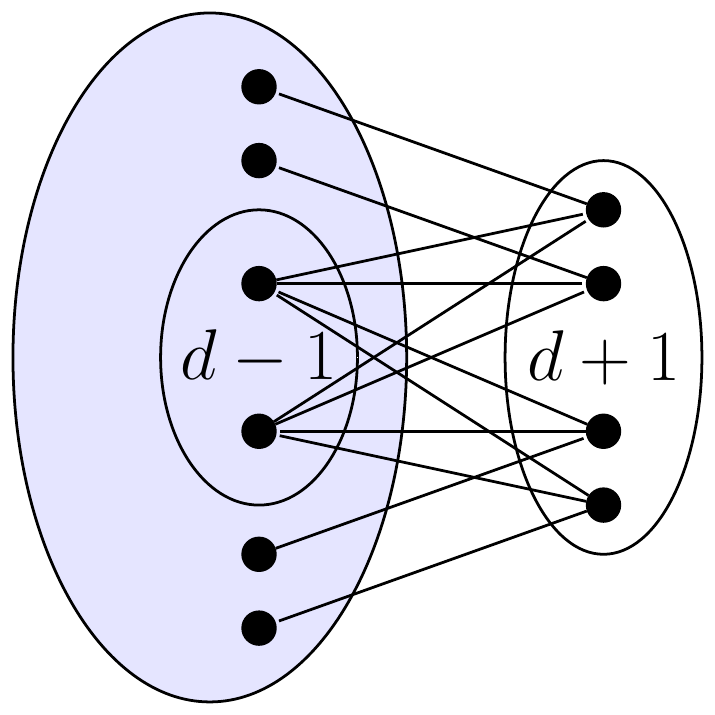}
  \caption{$G'_{n,d}$.}
\end{figure}

For $d=2$, we  conclude that $G\subseteq G'_{n,2}$, as claimed, and for $d \geq 3$, we get a contradiction
since $G'_{n,d}$ is hamiltonian.
This proves~\eqref{ni}.

\bigskip
Call a vertex $v \in D$ \emph{open} if it has at least two private neighbors, \emph{half-open} if it has exactly one private neighbor, and \emph{closed}
 if it has no private neighbors.

We say that {\em paths $P_1,\ldots,P_q$ partition $D$},
if these paths are vertex-disjoint and $V(P_1)\cup\ldots\cup V(P_q)=D$.
The idea of the proof is as follows: because $G - D$ is a complete graph, each path with endpoints in $G - D$ that covers all vertices of
 $D$ can be extended to a hamiltonian cycle of $G$.
So such a path does not exist, which implies that too few paths cannot partition $D$:

\begin{lem}\label{paths1}
If $s\geq 2$ then the minimum number of paths in $G[D]$ partitioning $D$ is at least $s$.
\end{lem}
\begin{proof}
Suppose $D$ can be partitioned into $\ell \leq s -1$ paths $P_1, \ldots, P_\ell$ in $G[D]$. Let $S = \{z_1, \ldots, z_s\}$. Then $P = z_1 P_1 z_2 \ldots z_{\ell} P_\ell z_{\ell + 1}$ is a path with endpoints in $V(G) - D$ that covers $D$. Because $V(G) - D$ forms a clique, we can find a $z_1,z_{\ell + 1}$ - path $P'$ in $G - D$ that covers $V(G) - D - \{z_2,\ldots, z_{\ell}\}$. Then $P \cup P'$ is a hamiltonian cycle of $G$, a contradiction.
\end{proof}

Sometimes, to get a contradiction with Lemma~\ref{paths1} we will use our information on vertex degrees in $G[D]$:

\begin{lem}\label{paths} Let $H$ be a graph on $r$ vertices such that for every nonedge $xy$ of $H$, $d(x) + d(y) \geq r-t$ for some $t$.
Then $V(H)$ can be partitioned into a set of at most $t$ paths. In other words, there exist $t$ disjoint paths $P_1, \ldots , P_t$ with $V(H) = \bigcup_{i=1}^t V(P_i)$.
\end{lem}

\begin{proof}
Construct the graph $H'$ by adding a clique $T$ of size $t$ to $H$ so that every vertex of $T$ is  adjacent to each vertex in $V(H)$. For each nonedge $x,y \in H'$,
$$d_{H'}(x) + d_{H'}(y) \geq (r-t) + t + t = r+t = |V(H')|.$$
 By Ore's theorem, $H'$ has a hamiltonian cycle $C'$.
Then $C' - T$ is a set of at most $t$ paths in $H$ that cover all vertices of $H$.
\end{proof}

The next simple fact will be quite useful.

\begin{lem}\label{open}If $G[D]$ contains an open vertex, then all other vertices are closed.
\end{lem}
\begin{proof}
Suppose $G[D]$ has an open vertex $v$ and another open or half-open vertex $u$. Let  $v',v''$ be some
private neighbors of $v$ in $S'$ and $u'$ be a neighbor of $u$ in $S'$.
By the maximality of $G$, graph $G+vu'$ has a hamiltonian cycle. In other
words, $G$ has a hamiltonian path $v_1v_2 \ldots v_n$, where $v_1=v$ and $v_n=u'$. Let $V'=\{v_i:  vv_{i+1}\in E(G)\}$. Since G has no hamiltonian cycle, $V'\cap
N(u')=\emptyset$.

Since  $d(v)+d(u')=n-1$, we have $V(G)=V'\cup N(u')+u'$. Suppose that $v'=v_i$ and $v''=v_j$.
Then $v_{i-1},v_{j-1}\in V'$, and $v_{i-1}, v_{j-1} \notin N(u')$. But among the neighbors of $v_i$ and $v_j$, only $v$
is not adjacent to $u'$, a contradiction.\end{proof}

Now we show that $S$ is non-empty and not too large.

\begin{lem}\label{de0} $ s \geq 1$.\end{lem}
\begin{proof} Suppose $S=\emptyset$.
If $D$ has an open vertex $v$, then by Lemma~\ref{open}, all other vertices are closed. In this case, $v$ is the only vertex of $D$ with neighbors outside of $D$, and hence $G \subseteq K'_{n,d}$, in which $v$ is the cut vertex.
Also if $D$ has at most one half-open vertex $v$, then similarly $G \subseteq K'_{n,d}$.

So suppose that $D$ contains  no open vertices but has two half-open vertices $u$ and $v$ with private neighbors $z_u$ and $z_v$ respectively.
Then $\delta(G[D])\geq d-1$. By P\' osa's Theorem, if $d\geq 4$, then $G[D]$ has a hamiltonian $v,u$-path. This path together with any
hamiltonian $z_u,z_v$-path in the complete graph $G-D$ and the edges $uz_u$ and $vz_v$ forms a hamiltonian cycle in $G$, a contradiction.

If $d=3$, then by Dirac's Theorem, $G[D]$ has a hamiltonian cycle, i.e. a $4$-cycle, say $C$. If we can choose our half-open $v$ and $u$
 consecutive on $C$, then $C-uv$ is a hamiltonian $v,u$-path in $G[D]$, and we finish as in the previous paragraph. Otherwise, we may
 assume that $C=vxuy$, where $x$ and $y$ are closed. In this case, $d_{G[D]}(x) = d_{G[D]}(y) = 3$, thus $xy\in E(G)$. So we again have
a hamiltonian $v,u$-path, namely $vxyu$, in $G[D]$. Finally, if $d=2$, then $|D| = 3$, and
$G[D]$ is either a $3$-vertex path  whose endpoints are half-open or a $3$-cycle. In both cases, $G[D]$ again has a hamiltonian path whose
ends are half-open.
\end{proof}

\begin{lem}\label{ded} $ s \leq d-3$.\end{lem}
\begin{proof} Since by~\eqref{s21}, $\delta(G)=d$, we have $s\leq d$. Suppose $s\in \{d-2,d-1,d\}$.

{\bf Case 1}: All vertices of $D$ have degree $d$.

{\em Case 1.1:} $s=d$.
 Then $G \subseteq H_{n,d+1}$. 

{\em Case 1.2:} $s=d-1$.
In this case, each vertex in graph $G[D]$ has degree $0$ or $1$.
By~\eqref{ni}, $G[D]$ induces a non-empty matching, possibly with some isolated vertices. Let $m$ denote the number of edges in $G[D]$.

If $m\geq 3$, then the number of components in $G[D]$ is less than $s$, contradicting Lemma~\ref{paths1}.
Suppose now  $m=2$, and the edges in the matching are $x_1y_1$ and $x_2y_2$. Then $d\geq 3$. If $d=3$, then $D=\{x_1,x_2,y_1,y_2\}$
and $G = F_{n,3}$ (see Fig~3 (right)). If $d\geq 4$, then
$G[D]$ has an isolated vertex, say $x_3$. This $x_3$ has a private neighbor
 $w\in S'$. Then $|S+w|=d$ which is
more than the number of components of $G[D]$ and we can construct a path from $w$ to $S$ visiting all components of $G[D]$.

Finally, suppose $G[D]$ has exactly one edge, say  $x_1y_1$. Recall that $d\geq 2$. Graph $G[D]$ has $d-1$ isolated vertices,
say $x_2,\ldots,x_d$. Each of $x_i$ for $2\leq i\leq d$
has a private neighbor $u_i$ in $S'$. Let $S=\{z_1,\ldots,z_{d-1}\}$.
 If $d=2$, then $S=\{z_1\}$, $N(D)=\{z_1,u_2\}$ and hence $G\subset H'_{n,2}$. So in this case the theorem holds for $G$.
If $d\geq 3$, then $G$ contains a path $u_dx_dz_{d-1}x_{d-1}z_{d-2}x_{d-2}\ldots z_2x_1y_1z_1x_2u_2$
 from $u_d$ to $u_2$ that covers $D$.

{\em Case 1.3:} $s=d-2$. Since $s\geq 1$, $d\geq 3$.
 Every vertex in $G[D]$ has degree at most $2$, i.e., $G[D]$ is a union of paths, isolated vertices, and cycles.
 Each isolated vertex has at least $2$ private neighbors in $S'$. Each endpoint of a path in $G[D]$  has one private neighbor in $S'$.
 Thus we can find disjoint paths from $S'$ to $S'$ that cover all isolated vertices and paths in $G[D]$ and all are disjoint from $S$.
Hence if the number $c$ of cycles in $G[D]$ is less than $d-2$, then we have a set of disjoint paths from $V(G)-D$ to $V(G)-D$ that
cover $D$ (and this set can be extended to a hamiltonian cycle in $G$). Since each cycle has at least $3$ vertices and $|D|=d+1$,
if $c\geq d-2$,
then  $ (d+1)/3 \geq d-2$, which is possible only when $d<4$, i.e. $d=3$. Moreover, then  $G[D]=C_3\cup K_1$ and $S=N$ is a single vertex.
But then $G=K'_{n,3}$.

{\bf Case 2}: There exists a vertex $v^* \in D$ with $d(v^*) = d+1$.
By Lemma~\ref{le0}~(b), $N = N(v^*) - D$, and so $G$ has at most one open or half-open vertex. Furthermore,
\begin{equation}\label{s11}
\parbox{5.4in}{\em
if $G$ has an open or half-open vertex, then it is $v^*$, and  by  Lemma~\ref{le0}, there are no other vertices of degree $d+1$.
}
\end{equation}

{\em Case 2.1:} $s = d $.  If $v^*$ is not closed, then it has a private neighbor $x \in S'$, and the neighborhood of each  other vertex of $D$ is
exactly $S$. In this case,  there exists a path from $x$ to $S$ that covers $D$. If $v^*$ is closed (i.e., $N = S$), then $G[D]$ has maximum degree $1$.
Therefore $G[D]$ is a matching with at least one edge (coming from $v^*$) plus some isolated vertices. If  this matching
has at least $2$ edges, then the number of components in $G[D]$ is less than $s$, contradicting Lemma~\ref{paths1}.
If $G[D]$ has exactly one edge, then $G \subseteq H'_{n,d}$.

{\em Case 2.2:} $s = d-1$.
  If $v^*$ is open, then $d_{G[D]}(v^*) = 0$ and  by~\eqref{s11}, each other vertex in $D$  has exactly one neighbor in $D$.
In particular, $d$ is even.  Therefore  $G[D-v^*]$ has $ d/2 $ components.
When $d \geq 3$ and $d$ is even, $ d/2  \leq s-1$ and we can find a path from $S$ to $S$ that covers $D - v^*$,
 and extend this path using two neighbors of $v^*$ in $S'$ to a path from $V(G)-D$ to $V(G)-D$ covering $D$.
Suppose $d=2$, $D=\{v^*,x,y\}$ and $S=\{z\}$. Then $z$ is a cut vertex separating $\{x,y\}$ from the rest of $G$, and hence $G\subseteq K'_{n,2}$.

If  $v^*$ is half-open, then by~\eqref{s11}, each other vertex in $D$ is closed and hence has exactly one neighbor in $D$. Let
  $x \in S'$ be the private neighbor of $v^*$. Then $G[D]$ is $1$-regular and therefore has exactly $ (d+1)/2 $ components, in particular,
  $d$ is odd.
 If $d \geq 2$ and is odd, then $(d+1)/2  \leq d-1 = s$, and so we can find a path from $x$ to $S$ that covers $D$.

 Finally, if $v^*$ is closed, then by~\eqref{s11}, every vertex of $G[D]$ is closed and has  degree $1$ or $2$, and $v^*$ has degree $2$ in $G[D]$.
 Then $G[D]$ has at most $\lfloor d/2\rfloor$ components, which is less than $s$ when $d \geq 3$. If $d = 2$, then $s=1$ and the unique vertex $z$
in $S$ is a cut vertex separating $D$ from the rest of $G$. This means $G\subseteq K'_{n,3}$.


{\em Case 2.3:} $s = d-2$. Since $s\geq 1$, $d\geq 3$.  If $v^*$ is open, then $d_{G[D]}(v^*) = 1$ and
by~\eqref{s11}, each other vertex in $D$  is closed and has exactly two neighbors in $D$.
 But this is not possible, since the degree sum of the vertices in  $G[D]$ must be even.
 If $v^*$ is half-open with a neighbor $x \in S'$, then $G[D]$ is $2$-regular. Thus $G[D]$ is a union of cycles
and has at most $\lfloor (d+1)/3 \rfloor$ components. When $d \geq 4$, this is less than $s$,  contradicting Lemma~\ref{paths1}.
If $d=3$, then $s=1$ and the unique vertex $z$
in $S$ is a cut vertex separating $D$ from the rest of $G$. This means $G\subseteq K'_{n,4}$.

If $v^*$ is closed, then $d_{G[D]}(v^*) = 3$ and  $\delta(G[D])\geq 2$.
So, for any  vertices $x,y$ in $G[D]$, \[d_{G[D]}(x) + d_{G[D]}(y) \geq 4 \geq (d+1) - (d-2-1) = |V(G[D])| - (s-1).\]
By Lemma~\ref{paths}, if $s\geq 2$, then  we can partition $G[D]$ into $s-1$ paths $P_1,..., P_{s-1}$.
This would contradict Lemma~\ref{paths1}. So suppose $s=1$ and $d=3$.
Then as in the previous paragraph, $G\subseteq K'_{n,4}$.
\end{proof}

Next we will show that we cannot have $2 \leq s \leq d-3$.
\begin{lem}\label{de1} $s = 1$.
\end{lem}

\begin{proof} Suppose $s = d-k$ where $3 \leq k \leq d-2$.

{\bf Case 1:} $G[D]$ has an open vertex $v$.
By Lemma~\ref{open}, every other vertex in $D$ is closed. Let $G' = G[D] - v$.
Then $\delta(G') \geq k-1$ and $|V(G')| = d$.
In particular, for  any   $x,y \in D-v$,
$$d_{G'}(x) + d_{G'}(y) \geq 2k-2 \geq k+1 = d - (d-k-1) = |V(G')| - (s-1).$$
By Lemma~\ref{paths}, we can find a path from $S$ to $S$ in $G$ containing all of $V(G')$.
Because $v$ is open, this path can be extended to a path from $V(G)-D$ to $V(G)-D$ including $v$,
and then extended to a hamiltonian cycle of $G$.

{\bf Case 2:} $D$ has no open vertices and $4 \leq k \leq d-2$.
Then $\delta(G[D]) \geq k-1$ and again for any   $x,y \in D$, $d_{G[D]}(x) + d_{G[D]}(y) \geq 2k-2$.
For $k\geq 4$, $2k-2 \geq k+2 = (d+1) - (d-k-1) = |D| - (s-1)$.
Since $k\leq d-2$,
by Lemma~\ref{paths}, $G[D]$ can be partitioned into $s-1$ paths, contradicting Lemma~\ref{paths1}.

{\bf Case 3}: $D$ has no open vertices and $s = d-3\geq 2$.
If there is at most one half-open vertex, then for any nonadjacent vertices $x,y \in D$, $d_{G[D]}(x) + d_{G[D]}(y) \geq 2 + 3 = 5 \geq (d+1) - (d-3-1)$,
and we are done as in Case 2.

So we may assume $G$ has at least 2 half-open vertices. 
Let $D'$ be the set of half-open vertices in $D$. If $D'\neq D$, let $v^*\in D-D'$.
Define a subset $D^-$  as follows: If $|D'|\geq 3$, then let $D^-=D'$, otherwise, let $D^-=D'+v^*$.
Let $G'$ be the  graph obtained from $G[D]$ by adding a new vertex $w$ adjacent to all vertices in $D^-$.
Then $|V(G')| = d+2$ and $\delta(G') \geq 3$.
In particular, for any $x,y \in V(G'), d_{G'}(x) + d_{G'}(y) \geq 6 \geq (d+2) - (d-3 - 1) = |V(G')|-(s-1)$.
By Lemma~\ref{paths}, $V(G')$ can be partitioned into $s-1$ disjoint paths $P_1, \ldots , P_{s-1}$.
We may assume that $w \in P_1$. If $w$ is an endpoint of $P_1$, then $D$ can also be partitioned into $s-1$
 disjoint paths $P_1 - w, P_2, \ldots , P_{s-1}$ in $G[D]$, a contradiction to Lemma~\ref{paths1}.

Otherwise, let $P_1 = x_1, \ldots , x_{i-1}, x_i, x_{i+1}, \ldots , x_k$ where $x_i = w$.
Since every vertex in $(D^-)-v^*$ is half-open and $N_{G'}(w)=D^-$,  we may assume that
 $x_{i-1}$ is half-open  and thus has a neighbor $y \in S'$. Let $S = \{z_1, \ldots,  z_{d-3}\}$. Then
\[y x_{i-1} x_{i-2} \ldots x_1 z_1 x_{i+1} \ldots x_k z_2 P_2 z_3  \ldots  z_{d-4} P_{d-4} z_{d-3}\] is a path in $G$ with endpoints in $V(G) - D$ that covers $D$.
\end{proof}

\medskip
Now we may finish the proof of Theorem \ref{newma}. By Lemmas~\ref{de0}--\ref{de1},
 $s = 1$, say, $S=\{z_1\}$. Furthermore, by Lemma~\ref{ded},
 \begin{equation}\label{pos}
 d\geq 3+s=4.
 \end{equation}

{\bf Case 1}: $D$ has an open vertex $v$. Then by Lemma~\ref{open}, every other vertex of $D$ is closed.
Since $s=1$, each $u\in D-v$  has degree  $d-1$
in $G[D]$. If $v$ has no neighbors in $D$, then $G[D] - v$ is a clique of order $d$, and $G \subseteq K'_{n,d}$. Otherwise,
since  $d\geq 4$, by Dirac's Theorem, $G[D]-v$ has a hamiltonian cycle, say $C$. Using $C$ and an edge from $v$ to $C$, we obtain
a hamiltonian path $P$ in $G[D]$ starting with $v$. Let $v' \in S'$ be a neighbor of $v$. Then $v'Pz_1$ is a path from $S'$ to $S$ that covers $D$, a contradiction.

{\bf Case 2}: $D$ has a half-open vertex but no open vertices. It is enough to prove that
\begin{equation}\label{hp}
\mbox{ $G[D]$ has
  a hamiltonian path $P$ starting with a half-open vertex $v$,}
 \end{equation}
since such a $P$ can be extended to a hamiltonian cycle in $G$ through $z_1$ and the private neighbor of $v$.
 If $d\geq 5$, then for any  $x,y\in D$,
 $$d_{G[D]}(x) + d_{G[D]}(y) \geq d-2 + d - 2 =2d-4\geq d+1 = |V(G[D])|.$$
 Hence by Ore's Theorem, $G[D]$ has a hamiltonian cycle, and hence~\eqref{hp} holds.

If $d<5$ then
by~\eqref{pos},
 $d=4$.  So $G[D]$ has $5$ vertices and minimum degree at least $2$.
By Lemma~\ref{paths}, we can find a hamiltonian path $P$ of $G[D]$, say $v_1v_2v_3v_4v_5$.
If at least one of $v_1,v_5$ is half-open or $v_1v_5\in E(G)$, then~\eqref{hp} holds. Otherwise,
each of $v_1,v_5$ has $3$ neighbors in $D$, which means $N(v_1)\cap D=N(v_5)\cap D=\{v_2,v_3,v_4\}$.
But then $G[D]$ has hamiltonian cycle $v_1v_2v_5v_4v_3v_1$, and again~\eqref{hp} holds.

{\bf Case 3}: All vertices in $D$ are closed.
Then $G \subseteq K'_{n,d+1}$, a contradiction. This proves the theorem.
%
%
%
%
%
\end{proof}

\section{A comment and a question}
\begin{itemize}
\item It was shown in Section~\ref{gen} that the right order of magnitude of $n_0(d,t)$ in Theorem \ref{sub} when $d=O(t)$
is $dt$. We can also show this when $d=O(t^{3/2})$. It could be that $dt$ is the right order of magnitude of $n_0(d,t)$ for all
$d$ and $t$.

\item
Is there a graph $F$ and positive integers $d$, $n$ with $n < n_0(d,t)$ and $d \leq \lfloor (n-1)/2 \rfloor$ such that for some $n$-vertex nonhamiltonian graph $G$ with
minimum degree at least $d$,
$$N(G,F)>\max\{N(H_{n,d}),F),N(K'_{n,d},F),N(H_{n,\lfloor (n-1)/2\rfloor},F)\}?$$

\end{itemize}

\end{document}